\documentclass[11pt]{amsart}
\usepackage{amsmath, amsthm, amssymb, latexsym,url}
\usepackage{hyperref}

\allowdisplaybreaks

\author{J. Vandehey}
\thanks{Email: \href{mailto:vandehey@uga.edu}{\nolinkurl{vandehey@uga.edu}}}
\title[Normal numbers for L\"uroth series]{A simpler normal number construction for simple L\"uroth series}
\date{\today}

\newtheorem{thm}{Theorem}[section]

\newtheorem{lem}[thm]{Lemma}

\newtheorem{prop}[thm]{Proposition}

\theoremstyle{definition}
\newtheorem{defn}[thm]{Definition}

\newcommand{\x}{\mathbf{x}}
\newcommand{\uu}{\mathbf{u}}
\newcommand{\p}{\mathbf{p}}
\newcommand{\m}{\mathbf{m}}

\begin{document}

\maketitle

\begin{abstract}
Champernowne famously proved that the number \[0.(1)(2)(3)(4)(5)(6)(7)(8)(9)(10)(11)(12)...\] formed by concatenating all the integers one after another is normal base $10$. We give a generalization of Champernowne's construction to various other digit systems, including generalized L\"uroth series with a finite number of digits. For these systems, our construction simplifies a recent construction given by Madritsch and Mance. Along the way we give an estimation of the sum of multinomial coefficients above a tilted hyperplane in Pascal's simplex, which may be of general interest.
\end{abstract}

\section{Introduction}

A number $x\in [0,1)$ with base $b$ expansion $x=0.d_1d_2d_3\dots$ is said to be normal base $b$ if for any string $s=a_1 a_2\dots a_k$ of base $b$ digits, we have
\[
\lim_{N\to \infty} \frac{\#\{ 0 \le n \le N-k\mid d_{n+i} = a_i, \quad 1 \le i \le k\}}{N} = b^{-k}.
\]
This may be interpreted as saying that for a normal number $x$, each digit string appears with the same relative frequency as every other digit string with the same length.

While many methods (most notably the Birkhoff Ergodic Theorem) can be used to show that almost all real numbers $x\in [0,1)$ are normal for any fixed base $b$, we know of very few examples of normal numbers. None of the well-known irrational constants, such as $e$ or $\pi$, are known to be normal to any base, and the only examples we have of normal numbers are those explicitly constructed to be normal. The first and still most famous of these constructions is Champernowne's constant \cite{Champernowne}, which in base $10$ looks like
\[
0.(1)(2)(3)(4)(5)(6)(7)(8)(9)(10)(11)(12)(13)\dots,
\]
formed by concatenating all the integers in succession. He derived this construction after proving the base $10$ normality of the following number
\begin{equation}\label{eq:champernowne}
0.(0)(1)(2)(3)(4)(5)(6)(7)(8)(9)(00)(01)(02)(03)\dots,
\end{equation}
formed by concatenating all base $10$ digit strings of length $1$ in lexicographical order, then all the digit strings of length $2$ in lexicographical order, and so on.

Constructions for base $b$ normal numbers usually fall into one of three methods: the combinatorial method first introduced by Copeland and Erd\H{o}s \cite{CE}, that is perhaps the most natural generalization of Champernowne's techniques; the exponential sum method first introduced by Davenport and Erd\H{o}s \cite{DE}; and the method of pseudo-random number generators used most powerfully by Bailey and Crandall \cite{BC1,BC2}.

Recently, mathematical interest has turned to providing constructions of normal numbers in other systems. In many cases, these proofs draw from techniques used by Champernowne, Copeland, and Erd\H{o}s. We shall be concerned here with ergodic fibred systems \cite{Schweiger}. Common examples of fibred systems include base $b$ expansions, continued fraction expansions, generalized L\"{u}roth series, and $\beta$-expansions. 

\begin{defn}Ergodic fibred systems consist with a transformation $T$ that maps a set $\Omega$ to itself, a measure $\mu$ on $\Omega$ that is finite and $T$-invariant, a digit set $\mathcal{D}\subset\mathbb{N}$, and a countable collection of disjoint subsets $\{I_d\}_{d\in \mathcal{D}}$ such that $\mu$-almost every point in $\Omega$ is in some $I_d$. The map $T$ is injective on each subset $I_d$ and $T$ is ergodic with respect to $\mu$.

The $T$-expansion of a point $x\in \Omega$ is then given by $x=[d_1,d_2,d_3,\dots]$ where $d_n$ is defined by $T^{n-1}x \in I_{d_n}$. For a given fibred system, we say a point $x\in\Omega$ with expansion $x=[d_1, d_2, d_3, \dots]$ is $T$-normal if for any string $s=[a_1, a_2, \dots, a_k]$ of digits from $\mathcal{D}$ we have
\begin{equation}\label{eq:mainlimit}
\lim_{N\to \infty} \frac{\#\{ 0 \le n \le N-k\mid d_{n+i} = a_i, \quad 1 \le i \le k\}}{N} = \mu(C[s]).
\end{equation}
where $C[s]$ is the cylinder set for the string $s$, i.e.,
\[
C[s]= \{x=[d_1, d_2, \dots] \mid d_i=a_i, \quad 1\le i \le k\}.
\]

We will often denote the measure of a cylinder set $s$ by $\lambda_s$, and if $s$ consists of a single digit $d$, then we will often shorthand the measure of the set $C[d]$ by $\lambda_d$.
\end{defn}

Madritsch and Mance \cite{MM} recently provided a normal number construction that works for many ergodic fibred systems, including those listed above. Their construction works roughly as follows:
\begin{enumerate}
\item Let $\epsilon_k$ be some small positive number shrinking to $0$ very quickly as $k$ increases, and let $S_k=\{s_1, s_2, s_3, \dots, s_n\}$ be a set enumerating all strings of length $k$ whose corresponding cylinder sets have measure at least $\epsilon_k$.
\item Let $M_k$ be at least $1/\epsilon_k$, and construct a string $X_k$ formed by concatenating first $\lfloor M_k \lambda_{s_1}\rfloor$ copies of $s_1$, then $\lfloor M_k \lambda_{s_2} \rfloor$ copies of $s_2$ and so on until ending with $\lfloor M_k \lambda_{s_n} \rfloor$ copies of $s_n$. By construction, we expect that for strings $s$ with length much smaller than $k$ that $s$ should appear in $X_k$ with close to the correct frequency.
\item We chose a quickly growing sequence $l_k$ and construct a digit $x$ by first concatenating $l_1$ copes of $X_1$, then $l_2$ copies of $X_2$, and so forth. The $l_k$'s are chosen so that $l_k$ copies of $X_k$ are vastly longer than the concatenated copies of $X_1$ up to $X_{k-1}$ that precede it.
\end{enumerate}
The strings $X_k$ are constructed to have better and better small-scale normality properties and then are repeated so many times in the construction of $x$ that their behavior swamps the behavior of what came before them. This construction was based on earlier work of Altomare and Mance \cite{AM}, and Mance \cite{M1,M2} independently. The construction also bears resemblence to an earlier, but less general construction of Martinelli \cite{Martinelli}, although their results appear to be independent.

The advantage of the Madritsch-Mance construction is that it is extremely general, working even for the notoriously difficult $\beta$-expansions. The disadvantage of the Madritsch-Mance construction is its inefficency. For example, if we apply the Madritsch-Mance construction to create a normal number base $10$, it, like Champernowne's secondary construction \eqref{eq:champernowne}, concatenates every digit string at some point; however, while Champernowne's second construction  uses each digit string exactly $1$ time, the Madritsch-Mance construction concatenates a string of length $k$ at least $k^{2k}\log k$ times.

Our goal in this paper is to construct and prove a much simpler normal number construction that, like Champernowne's construction, uses each digit string one time.

\begin{defn}\label{defn:xs}
Given an ergodic fibred system, let $S= \{s_n\}_{n\in \mathbb{N}}$ be an enumeration of all possible finite length strings ordered according to the following rule:
If $\lambda_{s_i} > \lambda_{s_j}$, then $i<j$. (We do not care how strings whose cylinder sets have the same measure are ordered compared to one another. Although, if we want a rigorous definition of $S$, we may impose a lexicographical order on these strings.)

Let $x_S$ be the point constructed by concatenating the strings $s_i$ in order.
\end{defn}

Note that if we consider a base $10$ fibred system and impose a lexicographical ordering on those strings in $S$ whose cylinder sets have the same measure, then we in fact get Champernowne's second construction \eqref{eq:champernowne} precisely. Therefore the construction of $x_S$ given in Definition \ref{defn:xs} is a true generalization of Champernowne's construction to more ergodic fibred systems.

Our goal in this paper will be to prove the following statement.

\begin{thm}\label{thm:main}
Consider an ergodic fibred system generated by a transformation $T$ such that $\mathcal{D}$ is finite and such that for each string $s=[a_1, a_2, \dots, a_k]$, we have $
\lambda_s \asymp \lambda_{a_1} \lambda_{a_2} \dots \lambda_{a_k}
$.

For such a system, the number $x_S$ constructed in Definition \ref{defn:xs} is $T$-normal.
\end{thm}

The simplest example of such a fibred system are the generalized L\"uroth series with finitely many digits, where we have, in fact, $\lambda_s = \lambda_{a_1} \lambda_{a_2} \dots \lambda_{a_k}$.
A good introduction to generalized L\"uroth series is given in  section 2.3 of \cite{DK}.

We note that for some fibred systems, there may not be a point $x\in \Omega$ with $T$-expansion given by $x_S$. This is due to the possibility of inadmissable strings, strings $s$ such that $\lambda_s=0$. $\beta$-expansions, in particular, have many inadmissable strings, and in the Madritsch-Mance construction, they get around this obstruction by including padding, a long, but finite string of $0$'s inserted before each concatenated string $s_i$. 

However, the condition in Theorem \ref{thm:main} that $
\lambda_s \asymp \lambda_{a_1} \lambda_{a_2} \dots \lambda_{a_k}
$ guarantees that no inadmissable strings exist. 

We leave as an open question---since we do not yet have enough information to be willing to state it as a conjecture---whether this construction works for other fibred systems, including Generalized L\"uroth Series with an infinite number of digits, continued fraction expansions, and (with an appropriate padding, \`a la Madritsch-Mance) $\beta$-expansions.

In the proof we shall make use of the following theorem, known alternately as Pjatetskii-Shapiro normality criterion or the hot spot theorem \cite{BM,MS}.

\begin{thm}[Pjatetskii-Shapiro]\label{thm:PS}
A point $x$ with expansion $x=[d_1, d_2, d_3, \dots]$ is $T$-normal if for any string $s=[a_1, a_2, \dots, a_k]$ we have
\[
\limsup_{N\to \infty} \frac{\#\{ 0 \le n \le N-k\mid d_{n+i} = a_i, \quad 1 \le i \le k\}}{N} \le c \cdot \lambda_s.
\] 
for some constant $c$ that is uniform over all strings.
\end{thm}

This normality criterion is quite useful because it means that instead of having to prove a precise asymptotic for the counting function on the left-hand side of \eqref{eq:mainlimit}, we need only know its value up to a constant multiple.

We will need another result on a sum of multinomial coefficients, which we present here.
Define the set $T_\epsilon$ by
\[
T_\epsilon =\left\{ \x = (x_1,\dots, x_D) \in \mathbb{R}^D\middle| \lambda_1^{x_1}\lambda_2^{x_2}\dots\lambda_D^{x_D}\ge \epsilon, x_i \ge 0, 1\le i \le D\right\}.
\]
We will use $\m=(m_1, m_2,\dots, m_D)\in \mathbb{Z}^D$ to denote an integer lattice point.
Then define
\begin{equation}\label{eq:S1def}
S(\epsilon)=\sum_{\m\in T_\epsilon} (m_1+m_2+\dots+m_D) \frac{(m_1+m_2+\dots+m_D)!}{m_1!m_2!\dots m_D!}
\end{equation}
and
\begin{equation}\label{eq:S2def}
S^\#(\epsilon) = \sum_{\m\in T_\epsilon} \frac{(m_1+m_2+\dots+m_D)!}{m_1!m_2!\dots m_D!}.
\end{equation}

\begin{thm}\label{thm:sbound}
We have
\[
S(\epsilon) \asymp \frac{| \log \epsilon|}{\epsilon} \qquad S^\#(\epsilon) \asymp \frac{1}{\epsilon}
\]
as $\epsilon$ tends to $0$.
\end{thm}

The proof of Theorem \ref{thm:main} will be broken down into the following steps.
\begin{enumerate}
\item In Section \ref{sec:counting}, we shall apply a counting argument to express
\[
\#\{ 0 \le n \le N-k\mid d_{n+i} = a_i, \quad 1 \le i \le k\} /N
\]
in terms of the sums $S(\epsilon)$ and $S^\#(\epsilon)$, so that Theorem \ref{thm:main} is a simple consequence of Theorem \ref{thm:sbound}.
\item In Sections \ref{sec:Hreduction} and \ref{sec:htos}, we will show that the bounds in Theorem \ref{thm:sbound} follow from bounds on similar sums, where $T_\epsilon$ is replaced by a hyperplane segment
\[
\mathcal{H}_\epsilon :=\left\{\x=(x_1, x_2, \dots, x_D) \middle| \sum_{i=1}^D x_i \log \lambda_i = \log \epsilon, \quad x_i\ge 0, \forall i\right\}.
\]
\item In Section \ref{sec:Hbound}, we analyze the size of the resulting sum over $\mathcal{H}_\epsilon$ by applying the Laplace method (see \cite{dB} for more details).
\end{enumerate}

In this paper we will frequently use Landau and Vinogradov asymptotic notations, such as $\ll$, $\gg$, $\asymp$, big-$O$, and little-$o$, all with the usual meanings.

\section{Some additional results}

We need a few general lemmas, which we will present here.

\begin{lem}\label{lem:gammagrowth}
Let $1<x<y$ and suppose that $0 < \delta < \min\{1,x-1\}$, then we have, uniformly in all variables
\[
\frac{\Gamma(y-\delta)}{\Gamma(x-\delta)} \ll \frac{\Gamma(y)}{\Gamma(x)} \ll \frac{\Gamma(y+\delta)}{\Gamma(x+\delta)} \quad \text{ and } \quad x\pm \delta \asymp x.
\]
\end{lem}

\begin{proof}
The first relation follows immediately from the fact that $\Gamma(x+\alpha)\asymp \Gamma(x) x^\alpha$ provided $x$ and $x+\alpha$ are on subset of the positive reals bounded away from $0$. The second relation is trivial.
\end{proof}

\begin{lem}\label{lem:inequality}
Let $n $ be a positive integer, $\{p_i\}_{i=1}^n$ be a set of real numbers, and $\{q_i\}_{i=1}^n$ be a set of positive numbers. Then we have that
\[
\frac{\left( \sum_{i=1}^n p_i\right)^2}{\sum_{i=1}^n q_i} \le \sum_{i=1}^n \frac{p_i^2}{q_i},
\]
with equality if and only if all the fractions $\{p_i/q_i\}_{i=1}^n$ have the same value.
\end{lem}

\begin{proof}
This follows immediately from the Cauchy--Schwarz inequality:
\[
\left(  \sum_{i=1}^n \sqrt{q_i} \cdot \frac{p_i}{\sqrt{q_i}} \right)^2 \le \left( \sum_{i=1}^n \sqrt{q_i}^2 \right) \left( \sum_{i=1}^n \left( \frac{p_i}{\sqrt{q_i}} \right)^2 \right)
\]
with equality if and only if there exists a constant $C$ such that $C \sqrt{q_i} = p_i/\sqrt{q_i}$.
\end{proof}

\begin{lem}\label{lem:sumtointegral}
For a fixed constant $C$, we have
\[
\sum_{-Z^{2/3}\le k \le Z^{2/3} } \exp\left( - \frac{C}{Z} k^2\right) = \sqrt{\frac{\pi Z}{C}}(1+o(1))
\]
as $Z$ tends to $\infty$.
\end{lem}

\begin{proof}
We apply Euler-Macluarin summation:
\begin{align*}
\sum_{-Z^{2/3}\le k \le Z^{2/3} } \exp\left( - \frac{C}{Z} k^2\right) &= \int_{-Z^{2/3}}^{Z^{2/3}}\exp\left( - \frac{C}{Z} x^2\right)\ dx\\
&\qquad + O\left(  \int_{-Z^{2/3}}^{Z^{2/3}} \frac{C|x|}{Z}\exp\left( - \frac{C}{Z} x^2\right)\ dx  \right)\\
&\qquad +O\left(\exp\left( - \frac{C}{Z} Z^{4/3}\right) \right)\\
&= \sqrt{\frac{\pi Z}{C}} - 2 \sqrt{\frac{Z}{C}} \cdot \int_{Z^{1/6}C^{1/2}}^{\infty} \exp\left( - x^2\right)\ dx\\
&\qquad + O\left(  \int_{0}^{Z^{1/6}C^{1/2}} |x|\exp\left( - x^2\right)\ dx  \right) +O\left(1 \right)\\
&=  \sqrt{\frac{\pi Z}{C}} (1+o(1))
\end{align*}
\end{proof}

\section{Proving Theorem \ref{thm:main} from Theorem \ref{thm:sbound}}\label{sec:counting}

By Theorem \ref{thm:PS}, it suffices to show that for any string $s=[a_1,a_2,\dots,a_k]$, we have
\[
\frac{\#\{ 0 \le n \le N-k\mid d_{n+i} = a_i, \quad 1 \le i \le k\}}{N}  \ll \lambda_s
\]
with implicit constant uniform over all strings.

The counting function 
\[
\#\{ 0 \le n \le N-k\mid d_{n+i} = a_i, \quad 1 \le i \le k\}
\]  is very difficult to compute directly, so we will instead estimate its size in terms of other, simpler functions. The $N$th digit of $x$, $d_N$, must appear in the concatenation of some string $s_n$, for which we have $\mu(C[s_n])= \epsilon = \epsilon(N)$.

Let $A(\epsilon;s)$ denote the number of time the string $s$ occurs \emph{within} the strings $s_i$ where $\lambda_{s_i} \ge \epsilon$.  Let $A(\epsilon)$ just denote the total number of digits in all the strings $s_i$ where $\lambda_{s_i}\ge \epsilon$. We will also use $A^\#(\epsilon)$ to denote the total number of strings $s_i$ where $\lambda_{s_i} \ge \epsilon$.

With $\epsilon=\epsilon(N)$, we clearly have 
\[
\#\{ 0 \le n \le N-k\mid d_{n+i} = a_i, \quad 1 \le i \le k\} \le A(\epsilon;s) + k A^\#(\epsilon)
\] 
where the latter term comes from a trivial estimate on how many times the string $s$ could occur starting in one string $s_i$ and ending another string $s_j$. Moreover, the number $N$ itself is at least $A(2\epsilon)$, and thus
\[
\frac{\#\{ 0 \le n \le N-k\mid d_{n+i} = a_i, \quad 1 \le i \le k\}}{N}  \le \frac{A(\epsilon;s) + k A^\#(\epsilon)}{A(2\epsilon)}.
\]

Now we wish to bound the $A$ functions, in terms of the $S$ functions \eqref{eq:S1def} and \eqref{eq:S2def}. Following the assumption from Theorem \ref{thm:main}, let us assume that for a string $s=[a_1,a_2, \dots, a_k]$ we have
\[
 c_1 \lambda_{a_1} \lambda_{a_2} \dots \lambda_{a_k} \le \lambda_s \le c_2 \lambda_{a_1} \lambda_{a_2} \dots \lambda_{a_k}.
\]

Suppose we want to count the total number of ways one can concatenate the string $s$ together with $m_d$ copies of the digit $d$. If counted with multiplicity, this will correctly count the total number of times $s$ occurs in strings that have $m_d+e_d$ copies of the digit $d$, where $e_d$ is the number of times $d$ occurs in $s$. There are precisely 
\[
(1+m_1+m_2+\dots+m_D)\cdot\frac{(m_1+m_2+\dots+m_D)!}{m_1!m_2!\dots m_D!}
\]
such strings (counted with multiplicity), each of which will have a cylinder set of measure in the interval
\[
\left[\frac{c_1}{c_2} \lambda_s \cdot \prod_{d\le D} \lambda_d^{m_d}, \frac{c_2}{c_1} \lambda_s \cdot \prod_{d\le D} \lambda_d^{m_d}\right]
\]
Thus if we let
\begin{align*}
S(\epsilon;s)&=\sum_{\substack{m_1, m_2, \dots m_D\\ \lambda_1^{m_1}\lambda_2^{m_2}\dots \lambda_D^{m_D}\ge \epsilon/\lambda_s  }} (1+m_1+m_2+\dots+m_D) \frac{(m_1+m_2+\dots+m_D)!}{m_1!m_2!\dots m_D!}\\
&= S(\epsilon/\lambda_s)+S^\#(\epsilon/\lambda_s),
\end{align*}
then we clearly have  
\[
 S\left( \frac{c_2}{c_1}\epsilon;s\right) \le A(\epsilon;s) \le S\left( \frac{c_1}{c_2}\epsilon;s\right).
\]

By a similar argument we can show
\[
 S\left( \frac{c_2}{c_1}\epsilon\right) \le A(\epsilon) \le S\left( \frac{c_1}{c_2}\epsilon\right) \qquad \text{and} \qquad S^\#\left( \frac{c_2}{c_1}\epsilon\right) \le A^\#(\epsilon) \le S^\#\left( \frac{c_1}{c_2}\epsilon\right)
\]

Thus,
\[
\frac{\#\{ 0 \le n \le N-k\mid d_{n+i} = a_i, \quad 1 \le i \le k\}}{N}  \le \frac{S\left(\frac{c_1}{c_2\lambda_s}\epsilon\right) + (k+1) S^\#\left(\frac{c_1}{c_2\lambda_s}\epsilon\right)}{S\left(2\frac{c_2}{c_1}\epsilon\right)}.
\]
Now applying Theorem \ref{thm:sbound} we obtain
\begin{align*}
\frac{\#\{ 0 \le n \le N-k\mid d_{n+i} = a_i, \quad 1 \le i \le k\}}{N} & \ll \frac{\left(\frac{c_1}{c_2\lambda_s}\epsilon\right)^{-1} \left|\log\left(\frac{c_1}{c_2\lambda_s}\epsilon\right)\right| + (k+1)\left(\frac{c_1}{c_2\lambda_s}\epsilon\right)^{-1}}{\left(2\frac{c_2}{c_1}\epsilon\right)^{-1} \left|\log\left(2\frac{c_2}{c_1}\epsilon\right)\right| }\\
&\ll \lambda_s
\end{align*}
and these bounds are uniform in $s$, which completes the proof of Theorem \ref{thm:main}

\section{Proof of Theorem \ref{thm:sbound}}\label{sec:Hreduction}

We will consider two new functions $H(\epsilon)$ and $H^\#(\epsilon)$ given by the following.

Let $\mathcal{H}_\epsilon$ denote the hyperplane segment
\[
\mathcal{H}_\epsilon :=\left\{\x=(x_1, x_2, \dots, x_D) \middle| \sum_{i=1}^D x_i \log \lambda_i = \log \epsilon, \quad x_i\ge 0, \forall i\right\}.
\]
Note that
\[
x_1 \log \lambda_1^{-1} + \dots + x_D \log \lambda_D^{-1} \le \log \epsilon^{-1}.
\]
is equivalent $\lambda_1^{x_1}\dots \lambda_D^{x_D} \ge \epsilon$. We will consider ``lattice'' points $\m\in \mathcal{H}_\epsilon$ to be given by $(m_1,m_2,\dots, m_D)$ where $m_2, \dots, m_D\in \mathbb{Z}$, and $m_1=M$ is a real number determined by the other coordinates via the formula 
\[
M= \frac{\log \left(\epsilon /\left( \lambda_2^{m_2}\lambda_3^{m_3}\dots\lambda_D^{m_D}\right) \right)}{\log \lambda_1}.
\]
We then define $H(\epsilon)$ and $H^\#(\epsilon)$ by
\begin{align*}
H(\epsilon) &:=\sum_{\m\in \mathcal{H}_\epsilon} (M+m_2+m_3+\dots+m_D) \frac{(M+m_2+m_3+\dots+m_D)!}{M!m_2!m_3!\dots m_D!}\\
H^\#(\epsilon) &:=\sum_{\m\in \mathcal{H}_\epsilon}  \frac{(M+m_2+m_3+\dots+m_D)!}{M!m_2!m_3!\dots m_D!}
\end{align*}
We extend the factorial to real values in the natural way by $x! = \Gamma(x+1)$.  

While the functions $S(\epsilon)$ and $S^\#(\epsilon)$ look at all values lying \emph{above} the hyperplane $\mathcal{H}_\epsilon$, the functions $H(\epsilon)$ and $H^\#(\epsilon)$ instead look at values \emph{on} the hyperplane $\mathcal{H}_\epsilon$.

Theorem \ref{thm:sbound} (and therefore Theorem \ref{thm:main}) will follow from the following two lemmas, which we prove in subsequent sections.

\begin{lem}\label{lem:htos}
We have
\[
H(\epsilon/\lambda_1) \ll S(\epsilon) \ll H(\epsilon\cdot \lambda_2) \qquad \text{ and } \qquad 
H^\#(\epsilon/\lambda_1) \ll S^\#(\epsilon) \ll H^\#(\epsilon\cdot \lambda_2).
\]
\end{lem}

\begin{lem}\label{lem:hbound}
We have
\[
H(\epsilon) \asymp \frac{| \log \epsilon|}{\epsilon} \qquad H^\#(\epsilon) \asymp \frac{1}{\epsilon}
\]
as $\epsilon$ tends to $0$.
\end{lem}

\section{Proof of Lemma \ref{lem:htos}}\label{sec:htos}

We shall provide bounds for $S(\epsilon)$. The method for $S^\#(\epsilon)$ is similar.

First, we place a lower bound on $S(\epsilon)$.  We have
\begin{align*}
S(\epsilon)&= \sum_{\substack{ m_2, \dots, m_D \\ \lambda_2^{m_2}\dots \lambda_D^{m_D}\ge \epsilon }}\left(\sum_{\substack{m_1\\ \lambda_1^{m_1}\lambda_2^{m_2}\dots \lambda_D^{m_D}\ge \epsilon }} (m_1+m_2+\dots+m_D) \frac{(m_1+m_2+\dots+m_D)!}{m_1!m_2!\dots m_D!}\right)\\
&\gg \sum_{\substack{ m_2, \dots, m_D \\ \lambda_2^{m_2}\dots \lambda_D^{m_D}\ge \epsilon}} (M'+m_2+\dots+m_D) \frac{(M'+m_2+m_3+\dots+m_D)!}{M'!m_2!m_3!\dots m_D!},
\end{align*}
where in each summand $M'$ is the largest integer such that
\begin{equation}\label{eq:M'def}
\lambda_1^{M'} \lambda_2^{m_2}\lambda_2^{m_3} \dots \lambda_D^{m_D}\ge \epsilon.
\end{equation}
Increasing the size of $\epsilon$ in the index of summation but not in the definition of $M'$ will only result in removing terms, therefore,
\begin{align*}
S(\epsilon)\gg \sum_{\substack{ m_2, \dots, m_D \\ \lambda_2^{m_2}\dots \lambda_D^{m_D}\ge \epsilon/\lambda_1}} (M'+m_2+\dots+m_D) \frac{(M'+m_2+m_3+\dots+m_D)!}{M'!m_2!m_3!\dots m_D!}.
\end{align*}
Comparing this series term by term with $H(\epsilon/\lambda_1)$ and noting that $M'$ for this sum is greater than and within $1$ of the corresponding $M$ in the terms of $H(\epsilon/\lambda_1)$, we get that $S(\epsilon)\gg H(\epsilon/\lambda_1)$ by Lemma \ref{lem:gammagrowth}.

For the reverse inequality, we have, for fixed $m_2, m_3, \dots, m_D$ and with $M'$ defined as in \eqref{eq:M'def}, that
\begin{align*}
&\sum_{\substack{m_1 \\ \lambda_1^{m_1} \lambda_2^{m_2} \dots \lambda_D^{m_D} \ge \epsilon}} (m_1+m_2+\dots+m_D) \frac{(m_1+m_2+\dots+m_D)!}{m_1!m_2!\dots m_D! }\\
&\qquad \le (M'+m_2+m_3+\dots+ m_D) \cdot \sum_{\substack{m_1 \\ \lambda_1^{m_1} \lambda_2^{m_2} \dots \lambda_D^{m_D}\ge \epsilon}} \frac{(m_1+m_2+\dots+m_D)!}{m_1!m_2!\dots  m_D!}\\
&\qquad = (M'+m_2+m_3+\dots+m_D) \cdot \frac{(m_2+m_3+\dots+m_D)!}{m_2!m_3!\dots m_D!} \times\\
&\qquad\qquad \times \sum_{\substack{m_1 \\ \lambda_1^{m_1} \lambda_2^{m_2} \dots \lambda_D^{m_D} \ge \epsilon}}  \binom{m_1+m_2+m_3+\dots+m_D}{m_2+m_3+\dots+m_D}\\
&\qquad = (M'+m_2+m_3+\dots+m_D) \cdot \frac{(m_2+m_3+\dots+m_D)!}{m_2!m_3!\dots m_D!} \times\\
&\qquad\qquad \times  \binom{M'+1+m_2+m_3+\dots+m_D}{1+m_2+m_3+\dots+m_D}\\
&\qquad  = \frac{m_2+1}{1+m_2+m_3+\dots+m_D} (M'+m_2+m_3+\dots+m_D)\times\\
&\qquad\qquad \times \binom{M'+(m_2+1)+m_3+m_4+\dots+m_D}{M, m_2+1,m_3,m_4,\dots,m_D}\\
&\qquad \ll (M'+(m_2+1)+m_3+\dots+m_D) \cdot \frac{(M'+(m_2+1)+m_3+m_4+\dots+m_D)!}{M! (m_2+1)!m_3!m_4!\dots m_D!}.
\end{align*}
By summing over all possible $m_2, m_3, \dots, m_D$ for which the sum is non-empty, we obtain most of the terms from $H(\epsilon\cdot \lambda_2)$, namely all the terms where $m_2 \ge 1$.  So therefore we have $S(\epsilon)\ll H(\epsilon\cdot \lambda_2)$.

\section{Proof of Lemma \ref{lem:hbound}}\label{sec:Hbound}

We shall provide the proof for $H(\epsilon)$ as the proof for $H^\#(\epsilon)$ is similar.

We want to begin by examining the terms of $H(\epsilon)$, using Stirling's formula.  We will use a somewhat non-standard form as follows:
\begin{equation}\label{eq:stirling}
x! \asymp \sqrt{2\pi (x+1)} \left(  \frac{x}{e}\right)^x.
\end{equation}
This clearly follows from the usual Stirling's formula for large $x$, since replacing $x$ by $x+1$ inside the square root introduces an error of at most $1+O(x^{-1})$; however this function has the added advantage of being true and uniform for all non-negative $x$, because the function on the right is bounded away from $0$.

Now consider a given term of $H(\epsilon)$,
\begin{align}\label{eq:exampleterm}
(M+m_2+m_3+\dots+m_D) \frac{(M+m_2+m_3+\dots+m_D)!}{M!m_2!m_3!\dots m_D!},
\end{align}
where, as before,
\[
M= \frac{\log \left(\epsilon /\left( \lambda_2^{m_2}\lambda_3^{m_3}\dots\lambda_D^{m_D}\right) \right)}{\log \lambda_1}.
\]
Applying Stirling's formula \eqref{eq:stirling} gives that \eqref{eq:exampleterm}  is on the order of $ G(\m) \cdot \exp\left( F(\m)   \right)$,
where 
\[
G(\m) := \frac{(M+m_2+\dots+m_D+1)^{3/2}}{\sqrt{(M+1)(m_2+1)(m_3+1)\dots (m_D+1)}}
\]
and
\begin{align*}
F(\m) &:= (M+m_2+m_3+\dots+m_D) \log(M+m_2+m_3+\dots+m_D)\\
&\qquad - M\log M- \sum_{i=2}^D m_i \log m_i.
\end{align*}
The function $G$ is fairly smooth and, compared to the exponential of $F$, quite small. Therefore we shall focus our studies primarily on understanding the properties of $F$.

\subsection{Understanding $F$}

In order to understand the properties of $F$ better, it is helpful to work with an auxiliary function.
Let 
\[
\tilde{F}(\x):=(x_1+\dots+x_D)\log (x_1+\dots+x_D) -\sum_{i=1}^D x_i \log x_i
\] 
be a function on $\mathcal{H}_\epsilon$.

We think of $F$ as being a function of $D-1$ variables. (The value of $m_1=M$ is determined by the others.) However, we will think of $\tilde{F}$ as a function on $D$ free variables, and then restrict our attention to the $D-1$-dimensional hyperplane $\mathcal{H}_\epsilon$.

\begin{prop}\label{prop:F2nd}
Let $l=(a_1 t+b_1, a_2 t + b_2, \dots, a_D t + b_D)$ be a line parallel to and intersecting the hyperplane segment $\mathcal{H}_\epsilon$. Then the second directional derivative of $\tilde{F}$ along this line is negative.
\end{prop}

\begin{proof}
Since $l$ is parallel to and intersecting $\mathcal{H}_\epsilon$, we have that
\[
\sum_{i=1}^D (a_i t + b_i) \log \lambda_i = \log \epsilon.
\]
By isolating the coeffecient of $t$, we obtain
\[
\sum_{i=1}^D a_i \log \lambda_i = 0.
\]
In particular, since all the $\log \lambda_i$ are negative, there must exist at least one positive and one negative $a_i$.

The second derivative of $F$ along this line is given by
\[
\frac{d^2}{dt^2} F(a_1 t+b_1, \dots, a_D t + b_D) = \frac{\left( \sum_{i=1}^D a_i \right)^2}{\sum_{i=1}^D \left( a_i t + b_i \right)} - \sum_{i=1}^D \frac{a_i^2}{a_i t + b_i}.
\]
By Lemma \ref{lem:inequality}, this is never positive, and is zero if and only if $a_i/(a_i t + b_i)$ has the same value for all $i$; however, in order to be in the domain of $F$, all the $a_i t + b_i$ must be positive, and as we noted earlier, at least one $m_i$ must be positive and at least one $m_i$ must be negative, therefore the $a_i / (a_i t + b_i)$ cannot all have the same value. The second derivative is therefore strictly negative.
\end{proof}

This proposition produces two immediate consequences. First, $\tilde{F}$ must have a unique local maximum on $\mathcal{H}_\epsilon$: it must have a maximum on $\mathcal{H}_\epsilon$ since it is a continuous function on a compact set, and there cannot be two local maximums since on the line between them $\tilde{F}$ would have strictly negative second derivative. Second, on any line passing through this maximum, the function $\tilde{F}$ is strictly decreasing away from the maximum.

\begin{lem}\label{lem:Fmax}
The function $\tilde{F}(\x)$ has its unique maximum on $\mathcal{H}_\epsilon$ at the point $\p=(\lambda_1 L, \lambda_2 L, \\ \lambda_3 L, \dots, \lambda_D L)$, where
\[
L=\frac{\log(\epsilon)}{\lambda_1 \log (\lambda_1) + \lambda_2 \log (\lambda_2) + \dots + \lambda_D \log (\lambda_D)}.
\]
Moreover, $\tilde{F}(\p) = - \log \epsilon$.
\end{lem}

\begin{proof}
It is easy to see that $\p$ is on the hyperplane segment $\mathcal{H}_\epsilon$. Since all the directional second derivatives parallel to $\mathcal{H}_\epsilon$ are negative, it suffices to show that, at the point $\p$, all the directional first derivatives parallel to $\mathcal{H}_\epsilon$ are $0$. 

As before, consider a line $l(t) = (a_1t + \lambda_1 L , \dots, a_n t + \lambda_D L)$ passing through the point $\p$. We again have
\[
\sum_{i=1}^D a_i \log \lambda_i = 0.
\]

The directional derivative of $F$ at $p$ along this line (in the positive $t$ direction) is given by
\begin{align*}
&\left( \sum_{i=1}^D a_i \right) \log \left(  \sum_{i=1}^D  \lambda_i L  \right) - \sum_{i=1}^D \left(  a_i \log ( \lambda_i L)  \right)\\
&\qquad =\left( \sum_{i=1}^D a_i \right) \log \left(  \frac{\log(\epsilon) \sum_{i=1}^D \lambda_i}{\sum_{i=1}^D \lambda_i \log \lambda_i}   \right) - \sum_{i=1}^D \left(  a_i \log \left(  \frac{\log(\epsilon) \lambda_i}{\sum_{j=1}^D \lambda_j \log \lambda_j}   \right)\right)\\
&\qquad =\left( \sum_{i=1}^D a_i \right) \log \left(  \frac{\log(\epsilon)}{\sum_{i=1}^D \lambda_i \log \lambda_i}   \right) - \sum_{i=1}^D \left(  a_i \log \left(  \frac{\log(\epsilon)}{\sum_{j=1}^D \lambda_j \log \lambda_j}   \right)\right)\\
&\qquad \qquad - \sum_{i=1}^D a_i \log \lambda_i\\
&\qquad = 0.
\end{align*} 

This shows that $\p$ is the maximum.  The value $\tilde{F}$ takes at this point is given by
\begin{align*}
&\left( \sum_{i=1}^D \lambda_i L \right) \log \left( \sum_{i=1}^D \lambda_i L \right) - \sum_{i=1}^D \lambda_i L \log (\lambda_i L) \\
&\qquad = \frac{\log \epsilon}{\sum_{j=1}^D \lambda_j \log \lambda_j} \log\left(\frac{\log \epsilon}{\sum_{j=1}^D \lambda_j \log \lambda_j} \right)\\
&\qquad\qquad - \sum_{i=1}^D \frac{\lambda_i \log \epsilon}{\sum_{j=1}^D \lambda_j \log \lambda_j} \log \left( \frac{\lambda_i \log \epsilon}{\sum_{j=1}^D \lambda_j \log \lambda_j} \right)\\
&\qquad = -\sum_{i=1}^D \frac{\lambda_i  \log \epsilon}{\sum_{j=1}^D \lambda_j \log \lambda_j } \log \lambda_i \\
&\qquad = - \log \epsilon
\end{align*}
which completes the proof.

\end{proof}

We will abuse notation and consider $\x\in \mathcal{H}_\epsilon$ as being both the vector $(x_1,x_2, \dots, x_D)$ and the vector $(x_2,\dots,x_D)$ with implied extra variable
\[
x_1 = \frac{1}{\log \lambda_1} \left( \log \epsilon - \sum_{i=2}^n x_i \log \lambda_i \right).
\]
And likewise we will consider $\p\in\mathcal{H}_\epsilon$ as being both the vector $(\lambda_1 L, \lambda_2 L, \dots, \lambda_D L)$ and the vector $(\lambda_2 L, \lambda_3 L, \dots, \lambda_D L)$.

Therefore $F$ can be given by
\begin{align*}
F(\x)&= \left(  \frac{\log \epsilon}{\log \lambda_1} +\sum_{i=2}^D x_i \left(1 - \frac{\log \lambda_i}{\log \lambda_1} \right) \right) \log  \left(  \frac{\log \epsilon}{\log \lambda_1} +\sum_{i=2}^D x_i \left(1 - \frac{\log \lambda_i}{\log \lambda_1}\right) \right) \\
&\qquad -  \left(  \frac{\log \epsilon}{\log \lambda_1} -\sum_{i=2}^D x_i\frac{\log \lambda_i}{\log \lambda_1}\right) \log  \left(  \frac{\log \epsilon}{\log \lambda_1} -\sum_{i=2}^D x_i \frac{\log \lambda_i}{\log \lambda_1}\right)\\
&\qquad - \sum_{i=2}^D x_i \log x_i.
\end{align*}

Given $2\le i,j\le D$, we have
\begin{align*}
\frac{\partial^2}{\partial x_i \partial x_j} F(\x) &=  \frac{\left(1- \frac{\log\lambda_i}{\log\lambda_1} \right) \left( 1- \frac{\log \lambda_j}{\log\lambda_1}\right)}{\frac{\log \epsilon}{\log \lambda_1} +\sum_{i=2}^D x_i \left(1 - \frac{\log \lambda_i}{\log \lambda_1}\right) }\\
&\qquad - \frac{\frac{\log \lambda_i \log \lambda_j}{(\log \lambda_1)^2}}{ \frac{\log \epsilon}{\log \lambda_1} -\sum_{i=2}^D x_i \frac{\log \lambda_i}{\log \lambda_1}} - \frac{\delta_{j,k}}{x_i }
\end{align*}
So, if we consider the second partial derivatives at $p$ arranged in a matrix, then we see that there exists a fixed real symmetric matrix $A$ (independent of $\epsilon$), such that
\[
\frac{\partial^2}{\partial x_i \partial x_j} F(\p) = \frac{1}{\log \epsilon} A_{i-1,j-1}. 
\]

Since $(\log \epsilon)^{-1} A$ is a real symmetric matrix, it can be diagonalized by orthogonal matrices. In particular, this implies that there exist unit vectors $\uu_2, \uu_3, \dots, \uu_D\in \mathbb{R}^{D-1}$ and fixed eigenvalues $l_2, l_3, \dots, l_D$  (again not dependent on $\epsilon$) such that
\[
\frac{\partial^2}{\partial \uu_j \partial \uu_D} F(\p) = \begin{cases}
\dfrac{l_j}{\log \epsilon} & \text{if }j=k\\
0 & \text{otherwise.}
\end{cases}
\]
By Proposition \ref{prop:F2nd} the second directional derivatives must always be negative, so $l_j$ must be positive.

Consider a ball $\mathcal{B}_\epsilon$ around the point $\p$, given by
\[
\mathcal{B}_\epsilon = \left\{ \p + t_2 \uu_2 + \dots + t_D \uu_D \middle| \sum_{i=2}^D t_i^2 \le |\log \epsilon|^{2/3} \right\}.
\]
 Note that for sufficiently small $\epsilon$, we have $\mathcal{B}_\epsilon \subset \mathcal{H}_\epsilon$.
We also consider a box $B_\epsilon$ given by
\[
B_\epsilon = \left\{ \p + t_2 \textbf{e}_2 + \dots + t_D \textbf{e}_D \middle| |t_i|\le \frac{1}{\sqrt{D-1}} |\log \epsilon|^{2/3} \right\},
\]
where $\textbf{e}_i$ are the elementary basis vectors. We have that $B_\epsilon\subset \mathcal{B}_\epsilon$.

If $\x\in \mathcal{B}_\epsilon$, then each coordinate $x_i$ of $\x$ must be on the order of $|\log x|$. Therefore for all points $\x\in \mathcal{B}_\epsilon$, the third partial derivative of $\tilde{F}$ satisfies the following bound:
\[
\frac{\partial^3}{\partial \uu_j \partial \uu_D \partial \uu_l} F(\x)  \ll |\log \epsilon|^{-2}.
\]

By Taylor's Theorem, for any point $\x = \p+t_2 \uu_2 + \dots + t_D \uu_D\in \mathcal{B}_\epsilon$, we have
\begin{equation}\label{eq:Ftaylor}
F(\x) = -\log \epsilon + \sum_{i=2}^D \frac{l_i}{\log \epsilon} t_i^2 + O(1).
\end{equation}
Let us let $F_+$ and $F_-$ be given by
\[
F_+(\x)= -\log \epsilon + \left( \max_{2\le i\le D}\frac{l_i}{\log \epsilon}\right)  \sum_{i=2}^D  t_i^2 
\]
and
\[
F_-(\x)= -\log \epsilon + \left( \min_{2\le i\le D}\frac{l_i}{\log \epsilon} \right) \sum_{i=2}^D  t_i^2 
\]
so that
\[
F_-(\x) \le \tilde{F}(\x)  + O(1) \le F_+(\x) .
\]
The advantage of these functions is that because $\sum_{j=2}^D t_j^2$ is invariant under rotating around $\p$. If $\x = \p + y_2 \textbf{e}_2 + \dots + y_D \textbf{e}_D$ is in the box $B_\epsilon$, then
\[
F_+(\x)= -\log \epsilon + \left( \max_{2\le i\le n}\frac{l_i}{\log \epsilon}\right)  \sum_{i=2}^n  y_i^2 
\]
and likewise for $F_-$.

Moreover, for each point $\x$ outside of the \emph{box} $B_\epsilon$, we can draw a line between $\x$ and $\p$ and note that by Lemma \ref{lem:Fmax}, $F$ increases along the line as we move towards $\p$. Therefore, the value of $F$ at $\x\not\in B_\epsilon$ is at most the maximum of $F$ on the boundary of $B_\epsilon$, and by \eqref{eq:Ftaylor}, this is at most $-\log \epsilon - C|\log \epsilon|^{1/3}$ for some fixed \emph{positive} constant $C$.

\subsection{Returning to the full sum}

For points $\x\in B_\epsilon$, it is easy to see that $
G(\x)
$
is on the order of $|\log \epsilon|^{(3-D)/2}$ and for $\x\not\in B_\epsilon$, the value $G(\x)$ could be as large as $|\log \epsilon|$. Therefore,
\begin{align*}
\sum_{\m\in\mathcal{H}_\epsilon\setminus B_\epsilon} G(\m) \exp(F(\m)) & \ll \sum_{\m\in\mathcal{H}_\epsilon\setminus B_\epsilon} |\log \epsilon| \exp(-\log \epsilon - C|\log \epsilon|^{1/3})\\
&\ll |\log \epsilon|^{D} \exp(-\log \epsilon - C|\log \epsilon|^{1/3}) = o (\epsilon^{-1})
\end{align*}
Here we used the fact that $m_i  \ll |\log \epsilon|$.

Therefore 
\[
H(\epsilon) \asymp \sum_{\m \in B_\epsilon} G(\m) \exp(\tilde{F}(\m)) +o(\epsilon^{-1}).
\]
Since, as noted above $
G(\m)
$
is on the order of $|\log \epsilon|^{(3-D)/2}$ for $\m\in B_\epsilon$, to complete the proof it suffices to prove that
\[
\sum_{\m \in B_\epsilon}  \exp(F(\m)) \asymp \frac{|\log \epsilon|^{D-1}}{\epsilon}
\]

First we note that
\[
 \sum_{\m\in B_\epsilon}  \exp(F_-(\m)) \ll \sum_{\m \in B_\epsilon}  \exp(F(\vec{m}))  \ll \sum_{\m\in B_\epsilon}  \exp(F_+(\m)) .
\]
There exists a point $\p'$ within distance $\sqrt{D-1}/2$ from $\p$, such that $\p'$ is an integer lattice point. For $\m\in B_\epsilon$, we have $F_\pm(\m) + F_\pm(\m')= O(|\log \epsilon|^{-1/3}) = O(1)$. Moreover, each vector $\m'=\m+\p-\p'$ can be written as $\p+k_2 \mathbf{e}_2 + \dots + k_D \mathbf{e}_D\in B_\epsilon$ with each $k_i$ in the interval $I=[-c|\log \epsilon|^{2/3}-\sqrt{D-1}/2, c|\log \epsilon|^{2/3}+\sqrt{D-1}/2]$. Therefore
\begin{align}
\sum_{\m \in B_\epsilon}  \exp(F_+(\m)) &\asymp \sum_{\m \in B_\epsilon}  \exp(F_+(\m+\p-\p'))\\
&\le \frac{1}{\epsilon} \prod_{i=2}^D \left(  \sum_{k_i \in I} \exp\left(  \left( \max_{2\le i\le D}\frac{l_i}{\log \epsilon}\right) k_i^2\right)  \right),\label{eq:fplus}
\end{align}
and likewise \begin{align}\label{eq:fminus}
\sum_{\m \in B_\epsilon}  \exp(F_-(\m)) \gg \frac{1}{\epsilon} \prod_{i=2}^D \left(  \sum_{k_i \in J} \exp\left(  \left( \max_{2\le i\le D}\frac{l_i}{\log \epsilon}\right) k_i^2\right)  \right),
\end{align}
where $J=[-c|\log \epsilon|^{2/3}+\sqrt{D-1}/2, c|\log \epsilon|^{2/3}-\sqrt{D-1}/2]$. Applying Lemma \ref{lem:sumtointegral} to \eqref{eq:fplus} and \eqref{eq:fminus} completes the proof.

\end{document}